\pdfoutput=1
\RequirePackage{ifpdf}
\ifpdf 
\documentclass[pdftex]{sigma}
\else
\documentclass{sigma}
\fi

\numberwithin{equation}{section}

\newtheorem{Theorem}{Theorem}[section]
\newtheorem*{Theorem*}{Theorem}

\newtheorem{Lemma}[Theorem]{Lemma}

{ \theoremstyle{definition}

	\newtheorem{Remark}[Theorem]{Remark} }
	
\newcommand{\CC}{\ensuremath{\mathbb{C}}}
\newcommand{\RR}{\ensuremath{\mathbb{R}}}
\newcommand{\ZZ}{\ensuremath{\mathbb{Z}}}

\newcommand{\DD}{\ensuremath{\mathbb{D}}}

\newcommand{\me}{\mathrm{e}}
\newcommand{\mo}{\mathcal{O}}

\begin{document}
\allowdisplaybreaks

\renewcommand{\thefootnote}{}

\newcommand{\arXivNumber}{2403.14942}

\renewcommand{\PaperNumber}{074}
	
\FirstPageHeading

\ShortArticleName{Asymptotics of the Humbert Function $\Psi_1$ for Two Large Arguments}
	
\ArticleName{Asymptotics of the Humbert Function $\boldsymbol{\Psi_1}$\\ for Two Large Arguments\footnote{This paper is a~contribution to the Special Issue on Asymptotics and Applications of Special Functions in Memory of Richard Paris. The~full collection is available at \href{https://www.emis.de/journals/SIGMA/Paris.html}{https://www.emis.de/journals/SIGMA/Paris.html}}}
	
\Author{Peng-Cheng HANG and Min-Jie LUO}

\AuthorNameForHeading{P.-C.~Hang and M.-J.~Luo}

\Address{Department of Mathematics, School of Mathematics and Statistics,\\Donghua University, Shanghai 201620, P.R.~China}

\Email{\href{mailto:mathroc618@outlook.com}{mathroc618@outlook.com}, \href{mailto:mathwinnie@live.com}{mathwinnie@live.com}, \href{mailto:mathwinnie@dhu.edu.cn}{mathwinnie@dhu.edu.cn}}

\ArticleDates{Received March 27, 2024, in final form August 02, 2024; Published online August 09, 2024}

\Abstract{Recently, Wald and Henkel (2018) derived the leading-order estimate of the Humbert functions $\Phi_2$, $\Phi_3$ and $\Xi_2$ for two large arguments, but their technique cannot handle the Humbert function $\Psi_1$. In this paper, we establish the leading asymptotic behavior of the Humbert function $\Psi_1$ for two large arguments. Our proof is based on a connection formula of the Gauss hypergeometric function and Nagel's approach (2004). This approach is also applied to deduce asymptotic expansions of the generalized hypergeometric function $_pF_q$ $(p\leqslant q)$ for large parameters, which are not contained in NIST handbook.}
		
		
\Classification{33C20; 33C65; 33C70; 41A60}

\renewcommand{\thefootnote}{\arabic{footnote}}
\setcounter{footnote}{0}

\section{Introduction}
Humbert \cite{Humbert-1922} introduced seven confluent hypergeometric functions of two variables which are denoted by $\Phi_1$, $\Phi_2$, $\Phi_3$, $\Psi_1$, $\Psi_2$, $\Xi_1$, $\Xi_2$. In this paper, we mainly focus on the Humbert function~$\Psi_1$, which is defined by
\[\Psi_1[a,b;c,c';x,y]=\sum_{m,n=0}^{\infty}\frac{(a)_{m+n}(b)_m}{(c)_m(c')_n}\frac{x^m}{m!}\frac{y^n}{n!},\qquad |x|<1,\quad |y|<\infty,\]
where $a,b\in\CC$ and $c,c'\notin\ZZ_{\leqslant 0}$. This function has a Kummer-type transformation \cite[equation~(2.54)]{Choi_Hasanov-2011}
\begin{equation}\label{Psi_1 Kummer transformation}
	\Psi_1[a,b;c,c';x,y]=(1-x)^{-a}\Psi_1\biggl[a,c-b;c,c';\frac{x}{x-1},\frac{y}{1-x}\biggr].
\end{equation}

Using the series manipulation technique, we can obtain \cite[equation~(83)]{Brychkov_Saad-2014}
\begin{equation}\label{Psi_1 series representation}
	\Psi_1[a,b;c,c';x,y]=\sum_{n=0}^{\infty}\frac{(a)_n}{(c')_n}{}_2F_1\biggl[\begin{matrix}
		a+n,b \\
		c
	\end{matrix};x\biggr]\frac{y^n}{n!},
\end{equation}
where $_2F_1$ denotes the Gauss hypergeometric function defined below in \eqref{pFq definition}. Similar analysis of \cite[equation~(14)]{Luo_Raina-2021} gives
\[\biggl|{}_2F_1\bigg[\begin{matrix}
	a+n,b \\
	c
\end{matrix};x\bigg]\biggr|=\mo\big(n^{-\omega}\rho_x^n\big),\qquad n\to\infty,\quad n\in\ZZ_{>0},\]
where
\[\omega=\min\{\operatorname{Re}(b),\operatorname{Re}(c-b)\},\qquad \rho_x=\max\bigl\{1,|1-x|^{-1}\bigr\}.\]
Then the summand in \eqref{Psi_1 series representation} has the order of magnitude
\[\mo\bigg(n^{\operatorname{Re}(a-c')-\omega}\frac{(\rho_x|y|)^n}{n!}\bigg),\qquad n\to\infty,\]
which implies that the series \eqref{Psi_1 series representation} converges absolutely in the region
\[\DD_{\Psi_1}:=\bigl\{(x,y)\in\CC^2\colon x\ne 1,\, |{\arg}(1-x)|<\pi,\, |y|<\infty\bigr\}.\]
So the series in \eqref{Psi_1 series representation} provides an analytic continuation of $\Psi_1$ to $\DD_{\Psi_1}$.

There are some useful identities about $\Psi_1$ in the literature (see \cite{Brychkov_Saad-2014,Choi_Hasanov-2011,ElHalba-2014,Humbert-1922}), as well as many applications in physics (see \cite[equation~(5.2)]{Ahbli_Mouayn-2018} and \cite{Belafhal_Saad-2017,Borghi-2023}). But we still know very little about the asymptotics of $\Psi_1$. Recently, in order to study the asymptotics of Saran's hypergeometric function $F_K$ when two of its variables become simultaneously large, Hang and Luo \cite{Hang_Luo-2024} established asymptotic expansions of $\Psi_1$ for one large variable. By using a Tauberian theorem for Laplace transform, Wald and Henkel \cite{Wald_Henkel-2018} derived the leading-order estimate of the Humbert functions $\Phi_2$, $\Phi_3$ and $\Xi_2$ when the absolute values of the two independent variables become simultaneously large. They also considered $\Psi_1$ and pointed out that their technique fails for $\Psi_1$ (see \cite[p.~99]{Wald_Henkel-2018}). In this paper, we give an incomplete answer to their problem by establishing the leading asymptotic behavior of $\Psi_1$ when $|x|\to \infty$ and $y\to +\infty$.

This paper is organised as follows. In Section~\ref{Sect. 2}, we demonstrate three lemmas which will be used later. Section~\ref{Sect. 3} devotes to the asymptotics of $\Psi_1$ for two large arguments. In Section~\ref{Sect. 4}, we present asymptotic expansions of the generalized hypergeometric function $_pF_q$ $(p\leqslant q)$ for large parameters, which are not contained in the NIST handbook \cite{NIST-Handbook}. The proofs in Sections \ref{Sect. 3} and~\ref{Sect. 4} are based on Nagel's approach \cite{Nagel-2004}. The main results are Theorems~\ref{Thm: Psi_1 two large arguments}, \ref{Thm: 2F2 large parameter--lambda}, \ref{Thm: 2F2 large parameter--n}, \ref{Thm: pFq large parameters} and~\ref{Thm: 2F2 other results}.

\textbf{Notation.}
In this paper, the number $C$ generically denotes a positive constant independent of the parameter $n$, the index of summation $\ell$ and the variable $z$. Moreover, the generalized hypergeometric function $_pF_q$ is defined by (see, for example, \cite[p.~404]{NIST-Handbook})
	\begin{equation}\label{pFq definition}
		{}_pF_q\bigg[\begin{matrix}
			a_1,\dots,a_p\\
			b_1,\dots,b_q
		\end{matrix};z\bigg]
		\equiv
		{}_pF_q[
		a_1,\dots,a_p;
		b_1,\dots,b_q;z]
		:=\sum_{n=0}^{\infty}\frac{(a_1)_n\cdots(a_p)_n}{(b_1)_n\cdots(b_q)_n}\frac{z^n}{n!},
	\end{equation}
	where $a_1,\dots,a_p\in\CC$ and $b_1,\dots,b_q\in\CC\setminus\ZZ_{\leqslant 0}$.

\section{Preliminary lemmas}\label{Sect. 2}
In this section, we deduce three lemmas which will be used in the sequel. The first is a sharp bound for the ratio of Pochhammer symbols.
\begin{Lemma}\label{Lemma: ratio of Pochhammer symbol}
	If $a\in\CC$ and $b\in\CC\setminus\ZZ_{\leqslant 0}$, then
	\[\biggl|\frac{(a)_n}{(b)_n}\biggr|\leqslant Cn^{\operatorname{Re}(a-b)},\qquad n\in\ZZ_{>0}.\]
\end{Lemma}
\begin{proof}
	The proof for $a\in\ZZ_{\leqslant 0}$ is trivial. If $a\in\CC\setminus\ZZ_{\leqslant 0}$, Stirling's formula implies that
		\[\lim_{n\to\infty}n^{b-a}\frac{\Gamma(a+n)}{\Gamma(b+n)}=1,\]
		which concludes that $n^{b-a}\frac{(a)_n}{(b)_n}$ is bounded uniformly for $n\in\ZZ_{>0}$.
\end{proof}

The second is a simple estimate of the function $\Phi_a(x)$, which is the ``horizontal'' generating function of Stirling numbers of real order (see \cite[Section~8]{Butzer_Kilbas_Trujillo-2003}). More properties of $\Phi_a(x)$, containing the asymptotic behavior for fixed $x\in\RR$ and $a\to\pm\infty$, are studied in \cite[Section~3.2]{VanGorder-2009}.
\begin{Lemma}\label{Phi_a(x) estimate}
	Define
	\[\Phi_a(x):=\sum_{k=1}^{\infty}\frac{k^a}{k!}x^k,\qquad a,x\in\RR.\]
	Then $\Phi_a(x)\sim x^a\me^x$, $x\to +\infty$. Thus
	\[\Phi_a(x)\leqslant Kx^a\me^x,\qquad x\geqslant 1,\]
	where $K>0$ is a constant independent of $x$.
\end{Lemma}
\begin{proof}
	Take $a_n=\frac{n^a}{n!}$ and $b_n=\frac{1}{n!}$ in \cite[p.~12, Problem~72]{Polya_Szego-vol.2}.
\end{proof}

The third is a global estimate for the confluent hypergeometric function $_1F_1$.

\begin{Lemma}\label{lem-1F1}
	Let $a,c\in\CC$. Choose $N\in\ZZ_{>0}$ such that $\operatorname{Re}(c+N+1)>0$ and define
	\begin{equation}\label{G_l definition}
			G_{\ell}(z):={}_1F_1\bigg[\begin{matrix}
				a+\ell\\
				c+\ell
			\end{matrix};z\bigg],\qquad \ell\in\ZZ_{\geqslant 0}.
	\end{equation}
	Then for $\ell\geqslant N+1$,
 \begin{equation}\label{G_l inequality}
			|G_\ell(z)|\leqslant C\me^{\gamma|z|},
	\end{equation}
	where $\gamma>0$ is a constant independent of $\ell$, $N$ and $z$.
\end{Lemma}

\begin{proof}
	Since $G_{\ell}(0)=1$, we can assume that $z\ne 0$. Recall the inequality \cite[equation~(2.3)]{Joshi_Arya-1982}
	\begin{equation}\label{1F1 inequality-1}
		\biggl|{}_1F_1\bigg[\begin{matrix}
			a_0\\
			b_0
		\end{matrix};z\bigg]\biggr|\leqslant \cos\frac{\theta}{2}\cdot
		{}_1F_1\biggl[\begin{matrix}
			|a_0|\\
			|b_0|
		\end{matrix};|z|\sec\frac{\theta}{2}\bigg]
	\end{equation}
	with $\theta=\arg(b_0)\in (-\pi,\pi)$ and the inequality \cite[p. 37, equation~(3.5)]{Carlson-1966}
	\begin{equation}\label{1F1 inequality-2}
		\me^{\frac{a_0}{b_0}z}<{}_1F_1\biggl[\begin{matrix}
			a_0\\
			b_0
		\end{matrix};z\biggr]<1-\frac{a_0}{b_0}+\frac{a_0}{b_0}\me^z,\qquad b_0>a_0>0,\quad z\ne 0.
	\end{equation}
	But when $a_0\geqslant b_0>0$ and $z>0$, we have $\frac{(a_0)_k}{(b_0)_k}\leqslant \big(\frac{a_0}{b_0}\big)^k$ since $\frac{a_0+j}{b_0+j}$ decreases with respect to~${j\geqslant 0}$. Thus
	\begin{equation}\label{1F1 inequality-3}
		{}_1F_1\biggl[\begin{matrix}
			a_0\\
			b_0
		\end{matrix};z\biggr]\leqslant\me^{\frac{a_0}{b_0}z},\qquad a_0\geqslant b_0>0,\quad z>0.
	\end{equation}
	Recall $\operatorname{Re}(c+N+1)>0$ and note that $\gamma(\ell):=\max\bigl\{1,\frac{|a+\ell|}{|c+\ell|}\bigr\}$ is bounded uniformly for $\ell\in\ZZ_{\geqslant 0}$. Therefore, a combination of the inequalities \eqref{1F1 inequality-1}--\eqref{1F1 inequality-3} claims that for $\ell\geqslant N+1$,
	\[|G_\ell(z)|\leqslant 2\gamma(\ell)\me^{\sqrt{2}\gamma(\ell)|z|}\leqslant C\me^{\gamma|z|},\]
	where
	\[\gamma:=\sup_{\ell\in\ZZ_{\geqslant 0}}\sqrt{2}\gamma(\ell)\geqslant \sqrt{2}.\]
	This completes the proof.
\end{proof}

\begin{Remark}
	Lemma \ref{lem-1F1} can be easily generalized to the following form. Let $a_1,\dots,a_p,b_1,\dots$, $b_p\in\CC$, and let $N\in\ZZ_{>0}$ such that $\operatorname{Re}(b_j+N+1)>0$, $1\leqslant j\leqslant p$. Then for $\ell\geqslant N+1$,
		\begin{equation}\label{pFp estimate}
			\biggl|{}_pF_p\biggl[\begin{matrix}
				a_1+\ell,\dots,a_p+\ell\\
				b_1+\ell,\dots,b_p+\ell
			\end{matrix};z\biggr]\biggr|\leqslant C\me^{\gamma|z|},
		\end{equation}
		where $\gamma>0$ is a constant independent of $\ell$, $N$ and $z$.
\end{Remark}

\section[Asymptotics of Psi\_1 for large arguments]{Asymptotics of $\boldsymbol{\Psi_1}$ for large arguments}\label{Sect. 3}
In this section, we establish the leading asymptotic behavior of $\Psi_1$ under the condition
\begin{gather}\label{Psi_1 asymptotic condition}
	x\to \infty,\qquad |{\arg}(1-x)|<\pi,\qquad y\to +\infty,\qquad \biggl|\frac{y}{1-x}\biggr|=\gamma
\end{gather}
satisfying $0<\gamma_1\leqslant \gamma\leqslant \gamma_2<\infty$.

First of all, we derive a new series representation for $\Psi_1$. Our starting point is the behavior near unit argument of the Gauss hypergeometric function, which is given by the well-known connection formula \cite[equation~(1.2)]{Buhring-1987}
\begin{align}
		\frac{\Gamma(a)\Gamma(b)}{\Gamma(c)}{}_2F_1\biggl[\begin{matrix}
			a,b \\
			c
		\end{matrix};z\biggr]= {}& \frac{\Gamma(a)\Gamma(b)\Gamma(s)}{\Gamma(a+s)\Gamma(b+s)}
		{}_2F_1\biggl[\begin{matrix}
			a,b \\
			1-s
		\end{matrix};1-z\biggr]\nonumber\\
		 & +\Gamma(-s)(1-z)^s
		{}_2F_1\biggl[\begin{matrix}
			a+s,b+s \\
			1+s
		\end{matrix};1-z\biggr]\label{2F1 unit argument}
\end{align}
with $|\arg z|<\pi$, $|{\arg}(1-z)|<\pi$ and $s=c-a-b$. Furthermore, \eqref{2F1 unit argument} is valid if $s\notin\ZZ$.

Expanding the right-hand side of \eqref{2F1 unit argument} as follows (see \cite[equation~(1.1)]{Buhring-1992}):
\begin{align}
		\frac{\Gamma(a)\Gamma(b)}{\Gamma(c)}{}_2F_1\biggl[\begin{matrix}
			a,b \\
			c
		\end{matrix};z\biggr]&{}= \sum_{n=0}^{\infty}(-1)^n\frac{\Gamma(a+n)\Gamma(b+n)\Gamma(s-n)}{\Gamma(a+s)\Gamma(b+s)n!}(1-z)^n \nonumber\\
		&\quad{} +\sum_{n=0}^{\infty}(-1)^n\frac{\Gamma(a+s+n)\Gamma(b+s+n)\Gamma(-s-n)}{\Gamma(a+s)\Gamma(b+s)n!}(1-z)^{n+s}\label{2F1 expansion}
\end{align}
and then applying \eqref{2F1 expansion} to \eqref{Psi_1 series representation}, we get
\begin{gather}
		\Psi_1[a,b;c,c';x,y] =C_1\sum_{n=0}^{\infty}\frac{(a)_n(b)_n}{(a+b-c+1)_n}
		{}_2F_2\biggl[\begin{matrix}
			a-c+1,a+n \\
			c',a+b-c+1+n
		\end{matrix};y\biggr]\frac{ (1-x )^n}{n!}\label{Psi_1 near x=1}\\
		\qquad{} +C_2 (1-x )^{c-a-b}\sum_{n=0}^{\infty}\frac{ (c-a )_n (c-b )_n}{ (c-a-b+1 )_n}
		{}_2F_2\biggl[\begin{matrix}
			a-c+1,a+b-c-n \\
			c',a-c+1-n
		\end{matrix};\frac{y}{1-x}\biggr]\frac{ (1-x )^n}{n!},\nonumber
\end{gather}
where $|{\arg}(1-x)|<\pi$, $a+b-c\notin\ZZ$,
\[C_1=\frac{\Gamma(c)\Gamma(c-a-b)}{\Gamma(c-a)\Gamma(c-b)},\qquad C_2=\frac{\Gamma(c)\Gamma(a+b-c)}{\Gamma(a)\Gamma(b)}
\]
and $_{2}F_{2}$ is defined in \eqref{pFq definition}.
Both series in \eqref{Psi_1 near x=1} converge absolutely for $|x-1|<1$ and $|y|<\infty$.

Combining \eqref{Psi_1 Kummer transformation} with \eqref{Psi_1 near x=1} gives the following series representation.

\begin{Theorem}
	Assume that $c,c'\notin\ZZ_{\leqslant 0}$ and $a-c\notin\ZZ$. Then when $a-b\notin\ZZ$,
	\begin{equation}\label{Psi_1 large arguments}
		\Psi_1[a,b;c,c';x,y]=\mathfrak{f}_c(b,a)(1-x)^{-a}V_1(x,y)+\mathfrak{f}_c(a,b)(1-x)^{-b}V_2(x,y)
	\end{equation}
	holds for $|{\arg}(1-x)|<\pi$, $|x-1|>1$ and $|y|<\infty$, where
	\begin{align}
		&V_1(x,y):=\sum_{n=0}^{\infty}\frac{(a)_n (c-b )_n}{ (a-b+1 )_n}
		{}_2F_2\biggl[\begin{matrix}
			a-c+1,a+n \\
			c',a-b+1+n
		\end{matrix};\frac{y}{1-x}\biggr]\frac{(1-x)^{-n}}{n!},\label{V_1(x,y) definition}\\
		&V_2(x,y):=\sum_{n=0}^{\infty}\frac{(b)_n (c-a )_n}{ (b-a+1 )_n}
		{}_2F_2\biggl[\begin{matrix}
			a-c+1,a-b-n\\
			c',a-c+1-n
		\end{matrix};y\biggr]\frac{(1-x)^{-n}}{n!},\label{V_2(x,y) definition}
	\end{align}
	and
	\[\mathfrak{f}_\gamma(a,b):=\frac{\Gamma(\gamma)\Gamma(a-b)}{\Gamma(a)\Gamma(\gamma-b)}.\]
\end{Theorem}

\begin{Remark}\quad
	\begin{itemize}\itemsep=0pt
		\item[(1)] The series (\ref{V_1(x,y) definition}) and (\ref{V_2(x,y) definition}) converge absolutely for $|x-1|>1$ and $|y|<\infty$.
		\item[(2)] When $s=c-a-b\in\ZZ$ in (\ref{2F1 unit argument}), the corresponding connection formulas of the Gauss hypergeometric function appear as \cite[equations~(15.3.10)--(15.3.12)]{Handbook-of-math-functions}. Thus if ${s=c-a-b\in\ZZ}$, one may derive the corresponding series representations of $\Psi_1$.
	\end{itemize}
\end{Remark}

Next we derive a uniform estimate of $_2F_2$ for large parameters by using Nagel's approach \mbox{\cite[equations~(A16)--(A19)]{Nagel-2004}}.

Define for $\varepsilon=\pm 1$ that
\begin{equation}\label{f_n(z) g_n(z) definition}
	f_n^{\varepsilon}(z):={}_2F_2\biggl[\begin{matrix}
		a,b+\varepsilon n\\
		c,d+\varepsilon n
	\end{matrix};z\biggr],\qquad g_n^{\varepsilon}(z):=\frac{\Gamma(b)\Gamma(d+\varepsilon n)}{\Gamma(d)\Gamma(b+\varepsilon n)}{}_2F_2\biggl[\begin{matrix}
		a,b\\
		c,d
	\end{matrix};z\biggr],
\end{equation}
where $_2F_{2}$ is given by \eqref{pFq definition}.

\begin{Lemma}\label{Lemma f_n-g_n estimate}
	Let $c\notin\ZZ_{\leqslant 0}$ and $d\notin\ZZ$. Then
	\begin{equation}\label{f_n-g_n upper bound}
		\bigl|f_n^{\varepsilon}(z)-g_n^{\varepsilon}(z)\bigr|\leqslant CB^n\cdot\max\bigl\{1,n^{\operatorname{Re}(d-b)}\bigr\}z^p\me^z,\qquad n\geqslant 1,\quad z\geqslant z_0,
	\end{equation}
	where
	\[B:=\sup_{m\in\ZZ}\biggl|\frac{b+m}{d+m}\biggr|\geqslant 1,\qquad p:=\operatorname{Re}(a-c)+\max\{0,\operatorname{Re}(b-d)\},\]
	and $z_0\geqslant 1$ chosen such that $z^p\me^z\geqslant 1$ holds for $z\geqslant z_0$.
\end{Lemma}

\begin{proof}
	The constant $B$ exists and satisfies $B\geqslant 1$, since $\lim_{|m|\to\infty}\frac{b+m}{d+m}=1$. Next, we just prove the inequality for $\varepsilon=-1$ since the proof for $\varepsilon=1$ is similar.
	
	Note that
	\[g_n^{\varepsilon}(z)=\frac{(1-b)_n}{(1-d)_n}\sum_{\ell=0}^{\infty}\frac{(a)_\ell(b)_\ell}{(c)_\ell(d)_\ell}\frac{z^{\ell}}{\ell!}.\]
	Since $q:=\operatorname{Re}(a+b-c-d)\leqslant p$, we get from Lemmas \ref{Lemma: ratio of Pochhammer symbol} and \ref{Phi_a(x) estimate} that for $z\geqslant z_0$,
	\begin{align}
		\bigl|g_n^{\varepsilon}(z)\bigr|&{} \leqslant Cn^{\operatorname{Re}(d-b)}+Cn^{\operatorname{Re}(d-b)}\sum_{\ell=1}^{\infty}\ell^{q}\frac{z^\ell}{\ell!}\leqslant Cn^{\operatorname{Re}(d-b)}(1+\Phi_q(z))\nonumber\\
		&{} \leqslant Cn^{\operatorname{Re}(d-b)}z^q\me^z\leqslant Cn^{\operatorname{Re}(d-b)}z^p\me^z.\label{g_n(z) upper bound}
	\end{align}
	
	Assume that $z\geqslant z_0$ and write
	\[f_n^{\varepsilon}(z)-1=\Biggl(\sum_{\ell=1}^{n}+\sum_{\ell=n+1}^{\infty}\Biggr)\frac{(a)_\ell(b-n)_\ell}{(c)_\ell(d-n)_\ell}\frac{z^\ell}{\ell!}=:S_1^*+S_2^*.\]
	If $1\leqslant \ell\leqslant n$, then
	\[\biggl|\frac{(b-n)_\ell}{(d-n)_\ell}\biggr|=\prod_{j=0}^{\ell-1}\biggl|\frac{b-n+j}{d-n+j}\biggr|\leqslant B^{\ell},\]
	and using Lemma \ref{Lemma: ratio of Pochhammer symbol} can get
	\begin{equation}\label{S_1^* estimate}
		|S_1^*|\leqslant C\sum_{\ell=1}^{n}B^{\ell}\ell^{\operatorname{Re}(a-c)}\frac{z^\ell}{\ell!}\leqslant CB^n\sum_{\ell=1}^{n}\ell^{\operatorname{Re}(a-c)}\frac{z^\ell}{\ell!}.
	\end{equation}
	If $\ell\geqslant n+1$, adopt Lemma \ref{Lemma: ratio of Pochhammer symbol} to obtain
	\[\biggl|\frac{(b-n)_\ell}{(d-n)_\ell}\biggr|=\biggl|\frac{(1-b)_n(b)_{\ell-n}}{(1-d)_n(d)_{\ell-n}}\biggr|\leqslant Cn^{\operatorname{Re}(d-b)}(\ell-n)^{\operatorname{Re}(b-d)}.\]
	Thus
	\[|S_2^*|\leqslant Cn^{\operatorname{Re}(d-b)}\sum_{\ell=n+1}^{\infty}\ell^{\operatorname{Re}(a-c)}(\ell-n)^{\operatorname{Re}(b-d)}\frac{z^\ell}{\ell!}.\]
	Since $1\leqslant \ell-n<\ell$, we have $(\ell-n)^{\operatorname{Re}(b-d)}\leqslant \ell^{\max\{0,\operatorname{Re}(b-d)\}}$. Therefore,
	\begin{equation}\label{S_2^* estimate}
		|S_2^*|\leqslant Cn^{\operatorname{Re}(d-b)}\sum_{\ell=n+1}^{\infty}\ell^{\operatorname{Re}(a-c)+\max\{0,\operatorname{Re}(b-d)\}}\frac{z^\ell}{\ell!}.
	\end{equation}
	
	Using Lemma \ref{Phi_a(x) estimate}, the inequality \eqref{f_n-g_n upper bound} follows from \eqref{g_n(z) upper bound}--\eqref{S_2^* estimate}.
\end{proof}

A direct application of Lemma \ref{Lemma f_n-g_n estimate} gives the following theorem.

\begin{Theorem}\label{Theorem V(z)-V^*(z) estimate}
	Let $\{v_n\}$ be a sequence satisfying $v_0=1$ and $v_n=\mo(\me^{rn})$, where $r>1$. Define
	\[V(z):=\sum_{n=0}^{\infty}v_n f_n^{\varepsilon}(z)z^{-n},\qquad V^*(z):=\sum_{n=0}^{\infty}v_n g_n^{\varepsilon}(z)z^{-n},\qquad z\ne 0,\]
	where $f_n^{\varepsilon}(z)$ and $g_n^{\varepsilon}(z)$ are given by \eqref{f_n(z) g_n(z) definition}, $c\notin\ZZ_{\leqslant 0}$ and $d\notin\ZZ$. Then
	\[V(z)-V^*(z)=\mo\big(z^{p-1}\me^z\big),\qquad z\to +\infty,\]
	where $p=\operatorname{Re}(a-c)+\max\{0,\operatorname{Re}(b-d)\}$.
\end{Theorem}

\begin{Remark}
	By substituting the asymptotic expansion of $_2F_2$ (see \cite[p.~380]{Paris-2005})
	\begin{equation}\label{2F2 large z}
		{}_2F_2\biggl[\begin{matrix}
			a,b+\nu\\
			c,d+\nu
		\end{matrix};z\biggr]\sim \frac{\Gamma(c)\Gamma(d+\nu)}{\Gamma(a)\Gamma(b+\nu)}z^{a+b-c-d}\me^z,\qquad z\to \infty,\quad|\arg z|<\pi/2
	\end{equation}
	into the series
	\[A(z):=\sum_{\nu=0}^{\infty}\frac{z^{-\nu}\Gamma(\alpha+\gamma+\nu)}{\nu!\Gamma(\beta+1+\nu)}{}_2F_2\biggl[\begin{matrix}
		\alpha+1,\beta+\nu\\
		\delta,\beta+1+\nu
	\end{matrix};-\frac{1}{z}\biggr],\]
	Jur\v{s}\.enas \cite[Section~4.2]{Jursenas-2014} claimed that
	\begin{equation}\label{A(z) asymptotics}
		A(z)\sim \frac{\Gamma(\alpha+\gamma)\Gamma(\delta)}{\Gamma(\alpha+1)\Gamma(\beta-\alpha-\gamma)}\mathrm{e}^{-1/z}(-z)^{\gamma+\delta},\qquad |z|\to 0,\quad\operatorname{Re}(z)<0.
	\end{equation}
	The reminder term in the expansion (\ref{2F2 large z}) is given in \cite{Lin_Wong-2018}, but it is not valid uniformly for large~$\nu$ and large~$z$. Thus, Jur\v{s}\.enas' expansion is not rigorous. Furthermore, Lemma \ref{Lemma f_n-g_n estimate} fails for the asymptotics of $A(z)$, so it is of interest to give a rigorous proof of (\ref{A(z) asymptotics}) in our future work.
\end{Remark}

We now state and prove the main result.

\begin{Theorem}\label{Thm: Psi_1 two large arguments}
	Assume that
	\[c,c'\notin\ZZ_{\leqslant 0},\qquad a-b,a-c\notin\ZZ,\qquad \operatorname{Re}(c-b)>0.\]
	Then under the condition \eqref{Psi_1 asymptotic condition},
	\begin{equation}\label{Thm: Psi_1 two large arguments 1}
		\Psi_1[a,b;c,c';x,y]\sim \frac{\Gamma(c)\Gamma(c')}{\Gamma(a)\Gamma(c-b)}\biggl(\frac{y}{1-x}\biggr)^b y^{a-2b-c'}\me^y.
	\end{equation}
\end{Theorem}

\begin{proof}
	Recall \eqref{Psi_1 large arguments} and note that $V_1(x,y)$ converges absolutely. Thus, the main contribution of $\Psi_1$ comes from $V_2(x,y)$. Using \eqref{2F2 large z} and Theorem \ref{Theorem V(z)-V^*(z) estimate}, we can obtain
	\begin{align}
		V_2(x,y)&{} ={}_1F_1\biggl[\begin{matrix}
			a-b\\
			c'
		\end{matrix};y\biggr]{}_1F_0\biggl[\begin{matrix}
			b\\
			-
		\end{matrix};\frac{1}{1-x}\biggr]+\mo\bigl(y^{p-1}\me^y\bigr)\nonumber\\
		&{} \sim \frac{\Gamma(c')}{\Gamma(a-b)}y^{a-b-c'}\me^y+\mo\bigl(y^{p-1}\me^y\bigr), \label{V_2(x,y) asymptotics}
	\end{align}
	where $p=\operatorname{Re}(a-c-c'+1)+\max\{0,\operatorname{Re}(c-b-1)\}$. Since
	\[\operatorname{Re}(a-b-c')>p-1 \ \Leftrightarrow \ \operatorname{Re}(c-b)>0,\]
	the result follows from \eqref{Psi_1 large arguments} and \eqref{V_2(x,y) asymptotics}.
\end{proof}

Numerical verification of Theorem \ref{Thm: Psi_1 two large arguments} is given in Appendix \ref{Appendix A}.

\section{Asymptotics of the generalized hypergeometric function}\label{Sect. 4}
Our derivation in Section~\ref{Sect. 3} depends on a rough estimate of $_2F_2$ for large $-n$. In this section, Nagel's approach is also used to explicitly establish the asymptotic behavior of $_2F_2$ for large parameters. We also present the asymptotic behavior of $_pF_q$ for large $-n$.

\subsection[Asymptotics of \_2F\_2]{Asymptotics of $\boldsymbol{_2F_2}$}
Let us examine the complete asymptotic expansions of $_2F_2$ for large parameters.

\begin{Theorem}\label{Thm: 2F2 large parameter--lambda}
	Let $\delta\in (0,\pi)$, $c\notin\ZZ_{\leqslant 0}$ and $z\ne 0$. Then for any positive integer $N$,
	\begin{equation}\label{2F2 large parameter--lambda}
		{}_2F_2\biggl[\begin{matrix}
			a,b+\lambda\\
			c,d+\lambda
		\end{matrix};z\biggr]=\sum_{k=0}^{N-1}\frac{(a)_k (d-b )_k}{(c)_k (d+\lambda )_k}\frac{ (-z )^k}{k!}{}_1F_1\biggl[\begin{matrix}
			a+k\\
			c+k
		\end{matrix};z\biggr]+\mo \bigl(\lambda^{-N} \bigr),
	\end{equation}
	where $\lambda\to\infty$ in the sector $|{\arg}(\lambda+d)|\leqslant \pi-\delta$.
\end{Theorem}

\begin{proof}
	If denoting
	\[v(z):={}_1F_1\biggl[\begin{matrix}
		a\\
		c
	\end{matrix};z\biggr],\qquad d_k:=\frac{(a)_k}{(c)_k k!},\]
	we can obtain
	\[{}_2F_2\biggl[\begin{matrix}
		a,b+\lambda\\
		c,d+\lambda
	\end{matrix};z\biggr]=\sum_{k=0}^{\infty}\frac{ (\lambda+b )_k}{ (\lambda+d )_k}d_k z^k.\]
	Note that the $k$-th derivative of $v(z)$ is given by \cite[p.~405, equation~(16.3.1)]{NIST-Handbook}
	\[v^{(k)}(z)=\frac{(a)_k}{(c)_k}
	{}_1F_1\biggl[\begin{matrix}
		a+k\\
		c+k
	\end{matrix};z\biggr].\]
	Applying Fields' result \cite[Theorem 3]{Fields-1967} yields
	\begin{gather}\label{2F2 transformation}
		{}_2F_2\biggl[\begin{matrix}
			a,b+\lambda\\
			c,d+\lambda
		\end{matrix};z\biggr]=\sum_{k=0}^{\infty}\frac{(d-b)_k}{(d+\lambda)_k}\frac{(-z)^k}{k!}v^{(k)}(z)=\sum_{k=0}^{\infty}\frac{(a)_k(d-b)_k}{(c)_k(d+\lambda)_k}\frac{(-z)^k}{k!}{}_1F_1\biggl[\begin{matrix}
			a+k\\
			c+k
		\end{matrix};z\biggr]
	\end{gather}
	with $\lambda+d\notin\ZZ_{\leqslant 0}$, and also gives the asymptotic expansion \eqref{2F2 large parameter--lambda}.
\end{proof}

\begin{Remark}
	~
	\begin{itemize}\itemsep=0pt
		\item[(1)] The series (\ref{2F2 transformation}) can be reformulated as follows:
		\begin{equation}\label{2F2 transformation reformulated}
			{}_2F_2\biggl[\begin{matrix}
				a,b\\
				c,d
			\end{matrix};z\biggr]=\sum_{k=0}^{\infty}\frac{(a)_k(d-b)_k}{(c)_k(d)_k}\frac{(-z)^k}{k!}{}_1F_1\biggl[\begin{matrix}
				a+k\\
				c+k
			\end{matrix};z\biggr],
		\end{equation}
		which is a specialization of Luke's result \cite[Section~9.1, equation~(34)]{Luke-1969}
		\begin{align}
					&{}_{p+1}F_{q+1}\biggl[\begin{matrix}
						a_1,\dots,a_p,b\\
						c_1,\dots,c_q,d
					\end{matrix};z\biggr]{}\nonumber\\
 & \qquad{}= \sum_{k=0}^{\infty}\frac{ (a_1 )_k\cdots (a_p )_k}{ (c_1 )_k\cdots (c_q )_k}\frac{(d-b)_k}{(d)_k}
 \frac{(-z)^k}{k!}{}_{p}F_{q}\biggl[\begin{matrix}
						a_1+k,\dots,a_p+k\\
						c_1+k,\dots,c_q+k
					\end{matrix};z\biggr],\label{pFp transformation}
\end{align}
			where $p\leqslant q$ with $z\in\CC$, or $p=q+1$ with $\operatorname{Re}(z)<\frac{1}{2}$.
		
		\item[(2)] Jur\v{s}\.enas \cite[Section~4.1, p.~69]{Jursenas-2014} mentioned an asymptotic formula
		\begin{equation}\label{Jursenas expansion-1}
			{}_2F_2\biggl[\begin{matrix}
				\alpha+1,\beta+\nu\\
				\delta,\beta+1+\nu
			\end{matrix};-\frac{1}{z}\biggr]\sim {}_1F_1\biggl[\begin{matrix}
				\alpha+1\\
				\delta
			\end{matrix};-\frac{1}{z}\biggr],\qquad \nu\in\ZZ_{\geqslant 0},\quad\nu\to \infty.
		\end{equation}
		without proof, whereas our Theorem \ref{Thm: 2F2 large parameter--lambda} gives a full asymptotic expansion of (\ref{Jursenas expansion-1}).
	\end{itemize}
\end{Remark}

The asymptotics of $_2F_2$ for $\lambda=-n$ is given below.
\begin{Theorem}\label{Thm: 2F2 large parameter--n}
	Let $c\notin\ZZ_{\leqslant 0}$, $d\notin\ZZ$, $b-d\notin\ZZ_{\geqslant 0}$ and $z\ne 0$. Then for any positive integer $N$,
	\begin{equation}\label{2F2 large parameter--n}
		{}_2F_2\biggl[\begin{matrix}
			a,b-n\\
			c,d-n
		\end{matrix};z\biggr]=\sum_{k=0}^{N-1}\frac{(a)_k(d-b)_k}{(c)_k(d-n)_k}\frac{(-z)^k}{k!}{}_1F_1\biggl[\begin{matrix}
			a+k\\
			c+k
		\end{matrix};z\biggr]+\mo\big(n^{-N}\big),
	\end{equation}
	where $n\to +\infty$ through integer values.
\end{Theorem}

\begin{proof}
	We shall follow Nagel's approach. For nonnegative integers $n$ and $\ell$, write
	\[F(z):={}_2F_2\biggl[\begin{matrix}
		a,b-n\\
		c,d-n
	\end{matrix};z\biggr],\qquad a_{\ell}(n):=\frac{(-1)^{\ell}}{(d-n)_{\ell}},\qquad g_{\ell}:=\frac{(a)_{\ell}(d-b)_{\ell}}{(c)_{\ell}\ell!}.\]
	Clearly, $\left|g_\ell\right|\leqslant C\ell^{\operatorname{Re}(\alpha)}$, where $\alpha:=a+d-b-c-1$. Moreover, \eqref{2F2 transformation reformulated} suggests that
	\begin{equation}\label{R(z) definition}
		R(z):=F(z)-G_0(z)=\sum_{\ell=1}^{\infty}a_\ell(n)g_\ell G_\ell(z)z^{\ell},
	\end{equation}
	where $G_\ell(z)$ is given by \eqref{G_l definition}.
	
	Take $n$ large, let $m=\lceil \log n\rceil$ and divide the series \eqref{R(z) definition} into five parts:
	\begin{equation}\label{five sums}
		R(z)=\sum_{1}^{N}+\sum_{N+1}^m+\sum_{m+1}^{n/2}+\sum_{n/2+1}^n+\sum_{n+1}^{\infty}=:S_0+S_1+S_2+S_3+S_4,
	\end{equation}
	where $N\in\ZZ_{>0}$ is chosen so that $\operatorname{Re}(c+N+1)>0$. Therefore, the inequality \eqref{G_l inequality} holds.
	
	Let us derive estimates for the sums in \eqref{five sums}.
	
	\textbf{Case 1.} $1\leqslant \ell\leqslant N$. Now $a_\ell(n)\sim n^{-\ell}$. Thus
	\[S_0=\sum_{\ell=1}^{N-1}a_\ell(n)g_\ell G_\ell(z)z^{\ell}+\mo\big(n^{-N}\big).\]
	
	\textbf{Case 2.} $N+1\leqslant\ell\leqslant m$. Now both $n$ and $n-\ell$ are large. The use of the identity $(z)_\ell=(-1)^{\ell}(1-z-\ell)_{\ell}$ gives
	\begin{equation}\label{a_l(n) identity}
		a_\ell(n)=\frac{1}{(1-d+n-\ell)_\ell}=\frac{\Gamma(1-d+n-\ell)}{\Gamma(1-d+n)}.
	\end{equation}
	By Stirling's formula, we get $a_\ell(n)\sim n^{-\ell}$ and thus
	\[|S_1|\leqslant C\mathrm{e}^{\gamma|z|}\sum_{\ell=N+1}^{m}\ell^{\operatorname{Re}(\alpha)}\biggl(\frac{|z|}{n}\biggr)^{\ell}.\]
	For $n$ large, $\ell^{\operatorname{Re}(\alpha)}\big(\frac{|z|}{n}\big)^{\ell}$ decreases with respect to $\ell\geqslant N+1$ and then
	\[|S_1|\leqslant C\mathrm{e}^{\gamma|z|}m(N+1)^{\operatorname{Re}(\alpha)}\biggl(\frac{|z|}{n}\biggr)^{N+1}=\mo\biggl(\frac{\log n}{n^{N+1}}\biggr).\]
	
	\textbf{Case 3.} $m+1\leqslant \ell\leqslant \frac{n}{2}$. Stirling's formula shows that
	\[a_\ell(n)=\frac{\Gamma(1-d+n-\ell)}{\Gamma(1-d+n)}\sim \biggl(1-\frac{\ell}{n}\biggr)^{-d}\frac{(n-\ell)!}{n!}.\]
	It follows from $\frac{1}{2}\leqslant 1-\frac{\ell}{n}<1$ that
	\[|a_\ell(n)|\leqslant C \frac{(n-\ell)!}{n!}=\frac{C}{(n-\ell+1)\cdots(n-1)n}\leqslant C\bigg(\frac{2}{n}\bigg)^{\ell}.\]
	As in Case 2, the monotonicity gives
	\[|S_2|\leqslant C\mathrm{e}^{\gamma|z|}\sum_{\ell=m+1}^{n/2}\ell^{\operatorname{Re}(\alpha)}\bigg(\frac{2|z|}{n}\bigg)^{\ell}\leqslant C\mathrm{e}^{\gamma|z|}nm^{\operatorname{Re}(\alpha)}\bigg(\frac{2|z|}{n}\bigg)^{m+1}=\mo\big(n^{-\frac{1}{2}\log n}\big).\]
	
	\textbf{Case 4.} $\frac{n}{2}+1\leqslant \ell\leqslant n$. Choose the least number $r\in\ZZ_{\geqslant 0}$ so that
	\[|(1-d+n-\ell)+r|\geqslant 4|z|>0.\]
	Note that $\frac{1}{|z+n|}$ is bounded uniformly for $n\geqslant 1$. Using \eqref{a_l(n) identity} gives
	\[|a_{\ell}(n)|=\prod_{j=0}^{\ell-1}\big|(1-d+n-\ell)+j\big|^{-1}\leqslant C(4|z|)^{r-\ell}.\]
	Thus
	\[|S_3|\leqslant C\mathrm{e}^{\gamma|z|}(4|z|)^r\sum_{\ell=n/2+1}^{n}\ell^{\operatorname{Re}(\alpha)}4^{-\ell}=\mo\big(n^{\operatorname{Re}(\alpha)+1}2^{-n}\big).\]
	
	\textbf{Case 5.} $\ell\geqslant n+1$. Now we can obtain $|a_\ell(n)|\leqslant C\frac{n^{\operatorname{Re}(d)}(\ell-n)^{1-\operatorname{Re}(d)}}{n!(\ell-n)!}$ from
	\[a_\ell(n)=\frac{(-1)^{\ell-n}}{(1-d)_n(d)_{\ell-n}}=\frac{(1)_n(1)_{\ell-n}}{(1-d)_n(d)_{\ell-n}}\frac{(-1)^{\ell-n}}{n!(\ell-n)!}.\]
	It follows that
	\begin{align*}
\begin{aligned}
		|S_4|&{} \leqslant C\mathrm{e}^{\gamma|z|}\frac{n^{\operatorname{Re}(d)}}{n!}\sum_{\ell=n+1}^{\infty}\ell^{\operatorname{Re}(\alpha)}(\ell-n)^{1-\operatorname{Re}(d)}\frac{|z|^{\ell}}{(\ell-n)!}\\
		&{} =C\mathrm{e}^{\gamma|z|}n^{\operatorname{Re}(d)}\frac{|z|^n}{n!}\sum_{j=1}^{\infty}(j+n)^{\operatorname{Re}(\alpha)}j^{1-\operatorname{Re}(d)}\frac{|z|^j}{j!}.
\end{aligned}
	\end{align*}
	Note that $\max\{n,j\}\leqslant j+n\leqslant 2\cdot\max\{n,j\}$. Then for $p,q\in\RR$, we have
	\begin{align*}
		\sum_{j=1}^{\infty}(j+n)^p j^q\frac{|z|^j}{j!}&{}\leqslant Cn^p\sum_{j=1}^{n}j^q\frac{|z|^j}{j!}+C\sum_{j=n+1}^{\infty}j^{p+q}\frac{|z|^j}{j!}\\
		&{} \leqslant Cn^{\max\{0,p\}}\sum_{j=1}^{\infty}j^{\max\{q,p+q\}}\frac{|z|^j}{j!},
	\end{align*}
	which shows that
	\[|S_4|\leqslant C\mathrm{e}^{\gamma|z|}n^{\max\{\operatorname{Re}(d),\operatorname{Re}(\alpha+d)\}}\frac{|z|^n}{n!}\sum_{j=1}^{\infty}j^{1-\operatorname{Re}(d)+\max\{0,\operatorname{Re}(\alpha)\}}\frac{|z|^j}{j!}.\]
	
	Now the asymptotic expansion \eqref{2F2 large parameter--n} follows from the estimates above.
\end{proof}

\begin{Remark}
	When $b-d\in\ZZ_{>0}$, the series (\ref{2F2 transformation reformulated}) terminates. Therefore, when $c\notin\ZZ_{\leqslant 0}$, $d\notin\ZZ$ and $b-d\in\ZZ_{>0}$, take $N=b-d+1$ in Theorem \ref{Thm: 2F2 large parameter--n}, and as a result, the asymptotic expansion of $_2F_2[a,b-n;c,d-n;z]$ is given by (\ref{2F2 large parameter--n}) with the error term vanishing.
\end{Remark}

\subsection[Asymptotics of \_pF\_q (p leqslant q)]{Asymptotics of $\boldsymbol{_pF_q}$ $\boldsymbol{(p\leqslant q)}$}

We have established the asymptotics of $_2F_2$ for large parameters. More generally, by using Nagel's approach, we can further obtain the following result about the generalized hypergeometric functions $_{p}F_{q}$ defined in \eqref{pFq definition}.

\begin{Theorem}\label{Thm: pFq large parameters}
\quad
	\begin{itemize}\itemsep=0pt	
		\item[$(1)$] Let $p$, $q$, $r$ and $s$ be nonnegative integers satisfying $q\geqslant p+1$ and $s\geqslant r-1$. Define
		\[\mathcal{F}_n^{(1)}(z):={}_{p+r}F_{q+s}\biggl[\begin{matrix}
			a_1-n,\dots,a_p-n,b_1,\dots,b_r\\
			c_1-n,\dots,c_q-n,d_1,\dots,d_s
		\end{matrix};z\biggr]\]
		and assume that $c_1,\dots,c_q\notin\ZZ$ and $d_1,\dots,d_s\notin\ZZ_{\leqslant 0}$. Then for any positive integer $N$,
		\[\mathcal{F}_n^{(1)}(z)=\sum_{k=0}^{N-1}\frac{(a_1-n)_k\cdots(a_p-n)_k}{(c_1-n)_k\cdots(c_q-n)_k}\frac{(b_1)_k\cdots(b_r)_k}{(d_1)_k
 \cdots(d_s)_k}\frac{z^k}{k!}+\mo\big(n^{(p-q)N}\big)\]
		as $n\to +\infty$ through integer values.
		
		\item[$(2)$] Let $p\in\ZZ_{\geqslant 0}$. Define
			\[\mathcal{F}_n^{(2)}(z):={}_{p+1}F_{p+1}\biggl[\begin{matrix}
				a_1,\dots,a_p,b-n\\
				c_1,\dots,c_p,d-n
			\end{matrix};z\biggr]\]
			and assume that $c_1,\dots,c_q\notin\ZZ_{\geqslant 0}$ and $b,d\notin\ZZ$. Then for any positive integer $N$,
			\[\mathcal{F}_n^{(2)}(z)=\sum_{k=0}^{N-1}\frac{(a_1)_k\cdots(a_p)_k}{(c_1)_k\cdots(c_p)_k}
 \frac{(d-b)_k}{(d-n)_k}\frac{(-z)^k}{k!}{}_{p}F_{p}\biggl[\begin{matrix}
				a_1+k,\dots,a_p+k\\
				c_1+k,\dots,c_p+k
			\end{matrix};z\biggr]+\mo\big(n^{-N}\big)\]
			as $n\to +\infty$ through integer values.
	\end{itemize}
\end{Theorem}

\begin{proof}
	The proof is much akin to that of Theorem \ref{Thm: 2F2 large parameter--n}, so it is sufficient to give some key estimates. To get (1), as in the proof of Theorem \ref{Thm: 2F2 large parameter--n}, establish the estimate
		\[\frac{(a_1-n)_\ell\cdots(a_p-n)_\ell}{(c_1-n)_\ell\cdots(c_p-n)_\ell}=\begin{cases}
			1+\mo(n^{-1}),& 1\leqslant\ell\leqslant N,\\
			\displaystyle 1+\mo\bigg(\frac{\log n}{n}\bigg),& N<\ell\leqslant m,\\
			\mo(1),& m<\ell\leqslant\frac{n}{2},\\
			 \mo(\omega^n),& \frac{n}{2}<\ell\leqslant n,\\
			 \mo\big(n^{\Delta}(\ell-n)^{-\Delta}\big),& \ell>n
		\end{cases}\]
		as $n\to +\infty$ through integer values, where $m=\lceil\log n\rceil$, $\Delta=\sum_{j=1}^{p}\operatorname{Re}(b_j-a_j)$ and $\omega>1$ is a~constant independent of $\ell$ and $n$. To get (2), use \eqref{pFp estimate} and \eqref{pFp transformation}. And the rest is the same.
\end{proof}

\begin{Remark}\quad
\quad
\begin{itemize}\itemsep=0pt
		\item[(1)] Knottnerus \cite{Knottnerus-1960} derived the asymptotic expansions of
		\[{}_pF_q\biggl[\begin{matrix}
			a_1+r,\dots,a_p+r\\
			b_1+r,\dots,b_q+r
		\end{matrix};z\biggr],\qquad r\to +\infty,\quad r\in\ZZ\]
		with $p\leqslant q+1$ and
		\[{}_{p+1}F_p\biggl[\begin{matrix}
			a_1+r,\dots,a_{k-1}+r,a_k,\dots,a_{p+1}\\
			b_1+r,\dots,b_k+r,b_{k+1},\dots,b_p
		\end{matrix};z\biggr],\qquad r\to +\infty,\quad r\in\ZZ\]
		with $1\leqslant k\leqslant p$. These results are also quoted and presented in \cite[Section~7.3]{Luke-1969}, \cite[Ap\-pendix~1]{Nagel-2004} and \cite[Section~16.11\,(iii)]{NIST-Handbook}. Our Theorem \ref{Thm: pFq large parameters} gives the full asymptotic expansion of $_pF_q$ $(p\leqslant q)$ for large $-n$, which does not appear in NIST handbook \cite{NIST-Handbook} and Luke's book \cite{Luke-1969}.
		
		\item[(2)] Nagel's approach cannot be applied to the asymptotics of
		\[{}_{p+1}F_p\biggl[\begin{matrix}
			a_1-n,\dots,a_{k-1}-n,a_k,\dots,a_{p+1}\\
			b_1-n,\dots,b_k-n,b_{k+1},\dots,b_p
		\end{matrix};z\biggr],\qquad n\to +\infty,\quad n\in\ZZ,\]
		where $1\leqslant k\leqslant p$, since the condition $\operatorname{Re}(b_1-n)>\operatorname{Re}(a_1-n)>0$ is needed for the Euler-type integral representation of $_{p+1}F_p$. We are interested in finding a more effective method than Nagel's approach.
	\end{itemize}
\end{Remark}

We end this section with the other results of $_2F_2$ for large parameters.

\begin{Theorem}\label{Thm: 2F2 other results}
	~
	\begin{itemize}\itemsep=0pt
		\item[$(1)$] Assume that $c,d\notin\ZZ$. Then for any positive integer $N$,
		\[{}_2F_2\biggl[\begin{matrix}
			a-n,b\\
			c-n,d-n
		\end{matrix};z\biggr]=\sum_{k=0}^{N-1}\frac{(a-n)_k(b)_k}{(c-n)_k(d-n)_k}\frac{z^k}{k!}+\mo\big(n^{-N}\big)\]
		as $n\to +\infty$ through integer values.
		
		\item[$(2)$] Assume that $a,b,c,d\notin\ZZ$. Then for any positive integer $N$,
			\[{}_2F_2\biggl[\begin{matrix}
				a-n,b-n\\
				c-n,d-n
			\end{matrix};z\biggr]=\me^z\sum_{k=0}^{N-1}\frac{(a-n)_k(d-b)_k}{(c-n)_k(d-n)_k}\frac{(-z)^k}{k!}{}_1F_1\biggl[\begin{matrix}
				c-a\\
				c-n+k
			\end{matrix};z\biggr]+\mo\big(n^{-N}\big)\]
			as $n\to +\infty$ through integer values.
	\end{itemize}
\end{Theorem}

\begin{proof}
	Assertion (1) follows immediately from Theorem \ref{Thm: pFq large parameters}\,(1). In order to prove Assertion~(2), we only need to note that combining \eqref{2F2 transformation reformulated} and the Kummer transformation \cite[equation~(13.2.39)]{NIST-Handbook} yields
		\[
		{}_2F_2\biggl[\begin{matrix}
			a-n,b-n\\
			c-n,d-n
		\end{matrix};z\biggr]
		=\mathrm{e}^z\sum_{k=0}^{\infty}\frac{(a-n)_k(d-b)_k}{(c-n)_k(d-n)_k}\frac{(-z)^k}{k!}{}_1F_1\biggl[\begin{matrix}
			c-a\\
			c-n+k
		\end{matrix};z\biggr].
		\]
		In addition, it is easy to verify that
		\[
		\biggl|{}_1F_1\biggl[\begin{matrix}
			c-a\\
			c-n+k
		\end{matrix};z\biggr]\biggr|\leqslant K,
		\]
		where $K$ is independent of $n$ and $k$. The rest of the proof is similar to that of Theorem \ref{Thm: pFq large parameters} and is omitted here.
\end{proof}

\appendix
\section{Numerical verification of Theorem \ref{Thm: Psi_1 two large arguments}}\label{Appendix A}

By using \textsc{Mathematica} 12.1, we provide a numerical verification of Theorem \ref{Thm: Psi_1 two large arguments}. The value of $\Psi_1$ is evaluated by using the following integral representation
	\[
	\Psi_1[a,b;c,c';x,y]
	=\frac{\Gamma(c)}{\Gamma(b)\Gamma(c-b)}\int_{0}^{1}t^{b-1}(1-t)^{c-b-1}(1-xt)^{-a}{}_{1}F_{1}\biggl[\begin{matrix}
		a\\
		c'
	\end{matrix};\frac{y}{1-xt}\biggr]\mathrm{d}t,
	\]
	where $a\in\mathbb{C}$, $\operatorname{Re}(c)>\operatorname{Re}(b)>0$, $c'\in\mathbb{C}\setminus\ZZ_{\leqslant 0}$, $y\in\mathbb{C}$ and $x\in\mathbb{C}\setminus[1,+\infty)$. The value of the right-hand side of \eqref{Thm: Psi_1 two large arguments 1} is denoted by $\mathrm{AE_{\Psi_1}}$. Tables \ref{Table 1} and \ref{Table 2} below clearly illustrate that the ratio $\frac{\Psi_1}{\mathrm{AE}_{\Psi_1}}$ approaches to $1$ as $x\rightarrow-\infty$ $(y=\gamma(1-x)\rightarrow+\infty)$.
\begin{table}[h]
	\renewcommand{\arraystretch}{1.2}
	\centering
	\caption{Numerical comparison when $c'=a=3$, $b=\frac{3}{2}$, $c=\frac{5}{2}$ and $\gamma=1$.}\label{Table 1}
\vspace{1mm}

	\begin{tabular}{|c|cc|}\hline\hline
		~ & $x$ & $\frac{\Psi_1}{\mathrm{AE}_{\Psi_1}}$ \\ \hline\hline
		1 & $-10$ & $1.06951$ \\ \hline
		2 & $-100$ & $1.00745$ \\ \hline
		3 & $-1000$ & $1.00075$ \\ \hline
		4 & $-2000$ & $1.00037$ \\ \hline
		5 & $-3000$ & $1.00025$
		\\ \hline\hline
	\end{tabular}
	\caption{Numerical comparison when $a=3$, $b=\frac{3}{2}$, $c=\frac{5}{2}$, $c'=2$ and $\gamma=\frac{1}{5}$.}\label{Table 2}
\vspace{1mm}

	\begin{tabular}{|c|cc|}\hline\hline
		~ & $x$ & $\frac{\Psi_1}{\mathrm{AE}_{\Psi_1}}$ \\ \hline\hline
		1 & $-10$ & $0.98215$ \\ \hline
		2 & $-100$ & $1.00223$ \\ \hline
		3 & $-1000$ & $1.00025$ \\ \hline
		4 & $-2000$ & $1.00012$ \\ \hline
		5 & $-3000$ & $1.00008$
		\\ \hline\hline
	\end{tabular}
\end{table}

\subsection*{Acknowledgements} The authors thank the referees for their valuable comments and suggestions.

\pdfbookmark[1]{References}{ref}
\LastPageEnding
		
\end{document}